\documentclass{amsart}
\usepackage{amsmath, amsthm, amssymb}
\usepackage{comment}

\newtheorem{Theo}{Theorem}
\newtheorem{prop}[Theo]{Proposition}
\newtheorem{cor}[Theo]{Corollary}
\newtheorem*{cornonumber*}{Corollary}
\newtheorem{lemma}[Theo]{Lemma}
\def\bt{\begin{Theo}}
\def\bl{\begin{lemma}}
\def\bp{\begin{proof}}
\def\et{\end{Theo}}
\def\el{\end{lemma}}
\def\ep{\end{proof}}

\theoremstyle{remark}
\newtheorem*{rem}{Remark}

\newcommand{\Z}{\mathbb{Z}}
\newcommand{\N}{\mathbb{N}}
\newcommand{\Q}{\mathbb{Q}}
\newcommand{\R}{\mathbb{R}}
\newcommand{\CC}{\mathbb{C}}
\newcommand{\F}{\mathcal{F}}
\newcommand{\OO}{\mathcal{O}}
\newcommand{\A}{\mathcal{A}}

\newcommand{\eps}{\varepsilon}
\newcommand{\frC}{\mathfrak{C}}

\newcommand{\f}{\frac}
\newcommand{\p}{\mathfrak{p}}
\newcommand{\e}{\mathfrak{e}}
\newcommand{\q}{\mathfrak{q}}
\newcommand{\NN}{\mathfrak{N}}
\newcommand{\m}{\mathfrak{m}}
\newcommand{\ff}{\mathfrak{f}}
\DeclareMathOperator{\sgn}{sgn}
\DeclareMathOperator{\reg}{Reg}
\DeclareMathOperator{\vol}{Vol}
\AtBeginDocument{\renewcommand{\a}{\mathfrak{a}}}
\AtBeginDocument{\renewcommand{\b}{\mathfrak{b}}}
\AtBeginDocument{\renewcommand{\c}{\mathfrak{c}}}
\AtBeginDocument{\renewcommand{\d}{\mathfrak{d}}}
\AtBeginDocument{\renewcommand{\O}{\mathcal{O}}}
\AtBeginDocument{\renewcommand{\P}{\mathcal{P}}}
\AtBeginDocument{\renewcommand{\l}{\ell}}

\def\cal#1{\mathcal{#1}}
\def\ba{\begin{align*}}
\def\ea{\end{align*}}
\def\frak#1{\mathfrak{#1}}
\def\half{{\f{1}{2}}}

\title[Explicit counting of ideals and a Brun-Titchmarsh inequality]{Explicit counting of ideals and a Brun-Titchmarsh inequality for the Chebotarev Density Theorem}

\begin{document}

\author{Korneel Debaene}
\address{K. Debaene, Georg-August-Universit\"at G\"ottingen, Bunsenstrasse 3-5, 37073 G\"ottingen, Germany}
\email{kdebaen@mathematik.uni-goettingen.de}

\begin{abstract} We prove a bound on the number of primes with a given splitting behaviour in a given field extension. This bound generalises the Brun-Titchmarsh bound on the number of primes in an arithmetic progression. The proof is set up as an application of Selberg's Sieve in number fields. The main new ingredient is an explicit counting result estimating the number of integral elements with certain properties up to multiplication by units. As a consequence of this result, we deduce an explicit estimate for the number of ideals of norm up to $x$.
\end{abstract}
\keywords{}
\subjclass[2010]{}

\maketitle

\section{Introduction and statement of results}
Let $q$ and $a$ be coprime integers, and let $\pi(x, q, a)$ be the number of primes not exceeding $x$ which are congruent to $a$ mod $q$. It is well known that for fixed $q$, \[\lim_{x\rightarrow \infty} \f{\pi(x, q, a)}{\pi(x)} = \frac{1}{\phi(q)}.\]
A major drawback of this result is the lack of an effective error term, and hence the impossibility to let $q$ tend to infinity with $x$. The Brun-Titchmarsh inequality remedies the situation, if one is willing to accept a less precise relation in return for an effective range of $q$. It states that \[\pi(x, q, a) \leq \frac{2}{\phi(q)}\frac{x}{\log x/q}, \quad \textup{ for any }x\geq q.\]
While originally proven with $2 + \eps$ in place of $2$, the above formulation was proven by Montgomery and Vaughan \cite{MontgomeryVaughan} in 1973. Further improvement on this constant seems out of reach, indeed, if one could prove that the inequality holds with $2-\delta$ in place of $2$, for any positive $\delta$, one could deduce from this that there are no Siegel zeros. However, improvements concerning the factor $\f{1}{\log x/q}$ have been made, and we refer to \cite{Maynard} for an overview of the state of the art.

A vast generalisation of the choice of a modulus $q$ and a residue class $a$, is the choice of a Galois extension $L/\Q$ with Galois group $G$, and a conjugacy class $C\subseteq G$. We then define \[\pi_C(x,L/\Q) =|\{p\leq x \mid p \textup{ unramified in $L/\Q$, Frob}_{L/\Q}(p)\in C \},\] the number of primes not exceeding $x$ with splitting behaviour in $L/\Q$ determined by $C$. The famous Chebotarev Density Theorem gives for fixed $L$ the asymptotic equivalence \[\pi_C(x,L/\Q) \sim \frac{|C|}{|G|}\textup{Li}(x).\]

An version of this result with effective error terms was first given by Lagarias and Odlyzko \cite{LO}, in the range $\log  x \geq 10[L:\Q](\log |\Delta_L|)^2$. Inspired by the Brun-Titchmarsh inequality, one can look for upper bounds which are off only by a constant factor, with a range of validity as wide as possible. In this respect, Lagarias, Montgomery and Odlyzko \cite{LMO} proved that \begin{equation}\label{eq:uptoconstanteq}\pi_C(x,L/\Q) \ll \frac{|C|}{|G|}\frac{x}{\log x},\end{equation} when $\log x \gg (\log |\Delta_L|) (\log_2 |\Delta_L|) (\log_3 |\Delta_L|)$. A recent paper by Thorner and Zaman \cite{TZ} gives a substantial improvement on the range where the inequality holds, using state-of-the-art information on the location of the zeros of Hecke $L$-functions. Consider an abelian subgroup $H\leq G$ such that $H\cap C$ is not empty, and let $K$ be the fixed field of $H$, of degree $n$. For a character $\chi\in \widehat{H}$, let $\ff_\chi$ be the conductor, and define \[\mathcal{Q}=\max_\chi \{N_{K/\Q}(\ff_\chi)  \mid \chi \in \widehat{H}\}.\] Their result is then that the inequality \eqref{eq:uptoconstanteq} holds in the range $x\geq 2|\Delta_K|^{163}\mathcal{Q}^{123} + |\Delta_K|^{55}\mathcal{Q}^{87}n^{68n}$. They also proved a stronger version, where the unspecified constant factor in \eqref{eq:uptoconstanteq} is replaced by $2$ plus an error term, which is valid in a similar but slightly smaller range.

The main result in this paper is the following explicit bound on $\pi_C(x,L/\Q)$. Let $\Delta_K$ be the discriminant of $K$, $h_K$ the class number of $K$, $\textup{Reg}_K$ the regulator of $K$ and $\kappa_K$ the residue of the Dedekind zeta function $\zeta_K(s)$ at $s=1$.
\bt\label{maintheorem} Let $L/\Q$ be a Galois extension with Galois group $G$, and let $C\subseteq G$ be a conjugacy class. Let $H$ be an abelian subgroup of $G$ such that $H\cap C \neq \emptyset$, and let $K$ be the fixed field of $H$, of degree $n$. Let $\ff \subseteq \O_K$ be the conductor of $L/K$. Define the constant \[c_K = n^{31n^3}(\textup{Reg}_Kh_K)^3(1+\log \textup{Reg}_Kh_K)^{3(n-1)^2} [L:K]^nN(\ff) \kappa_K^{-2n}. \] Then, for all $x$ with $x(\log x)^{-5n^2}\geq c_K$, we have 
\[\pi_C(x,L/\Q) \leq \frac{|C|}{|G|}\left(2+\frac{41n^2\log\log(x)}{\log(x/c_K)}\right)\frac{x}{\log(x/c_K)}.\]
\et
\begin{rem} In terms of the discriminant $\Delta_K$, our result gives a wider range than the previous results. By the analytic class number formula, $\textup{Reg}_Kh_K$ is interchangeable with $\sqrt{|\Delta_K|}\kappa_K$ up to a factor bounded in terms of $n$. Stark \cite{Stark} gave an effective (but unfortunately not explicit) lower bound for the residue, namely that there is an effective absolute constant $c$ such that $\kappa_K \geq \frac{c}{(n!)n|\Delta_K|^{1/n}}$, where $n|\Delta_K|^{1/n}$ can be replaced by $\log|\Delta_K|$ if $K$ has no quadratic subfield, and $n!$ can be dropped if $K$ has a normal tower over $\Q$. So all in all, our result gives a Brun-Titchmarsh bound in the range $x\gg_{n,[L:K],N(\ff),\eps} |\Delta_K|^{7/2 - 3/n +\eps}$. In situations where the discriminant of $K$ dominates the other factors, this is superior to the hitherto widest known range given by \cite{TZ}.\end{rem}


The proof follows a completely different path than the previous results mentioned on $\pi_C(x,L/\Q)$. In particular, we stress that we do not use any analytic information on zeta functions or their zeros. The proof is built up in the following stages. First and foremost, we develop an explicit counting method to count the number of integral elements, with certain properties, of norm up to $x$ and up to multiplication by units. This is achieved in Theorem \ref{keylemma}, by translating this counting problem to a lattice point counting problem and applying a Lipschitz argument. We take particular care to minimize the field constants appearing, by employing the so-called Schmidt's partition trick, and by carefully treating the occurring successive minima. The utility of this explicit counting result will be twofold. Firstly, it allows us to set up an explicit Selberg sieve in number fields, and secondly, through Artin's reciprocity law, the primes with given splitting behaviour can be expressed as integral elements with certain properties, which our counting result can estimate.

To illustrate the strength of our counting result, we mention that one natural and immediate application is the following explicit estimate for the number of ideals of bounded norm. We will use the practical notation $f=O^{*}(g)$ which means that $|f|\leq g$.

\begin{cor}\label{explicitideals} Let $K$ be a number field of degree $n$, and let $a_k$ be the number of ideals of norm $k$. Then, for all $x\geq 1$, \[\sum_{k\leq x}a_k = \kappa_K x + O^{*}\left(n^{10n^2}(\textup{Reg}_Kh_K)^{1/n}(1+ \log\textup{Reg}_Kh_K)^{\frac{(n-1)^2}{n}} x^{1-\frac{1}{n}}\right).\]
Moreover, let $\mathfrak{c}$ be an element of the class group of $K$, and let $b_k$ be the number of ideals of norm $k$ which lie in the class $\mathfrak{c}$. Then,
\[\sum_{k\leq x}b_k = \frac{\kappa_K}{h_K} x + O^{*}\left(n^{10n^2}(\textup{Reg}_K)^{1/n}(1+ \log\textup{Reg}_K)^{\frac{(n-1)^2}{n}} x^{1-\frac{1}{n}}\right).\]
\end{cor}
\begin{rem}One may again rephrase the constant Reg$_Kh_K$ in terms of the discriminant $\Delta_K$, using the analytic class formula and the bound by Louboutin \cite{Loubas} $\kappa_K\leq (\frac{e\log|\Delta_K|}{2(n-1)})^{n-1}$, to bound the error term by $\ll_{n,\eps} |\Delta_K|^{\f{1}{2n} + \eps}x^{1-\f{1}{n}}$. Thus one can get a meaningful estimate for $x$ as small as $\gg_{n,\eps} |\Delta_K|^{1/2+\eps}$.
\end{rem}

The quality of the error term is to be compared with what zeta function methods would yield. One can derive roughly the same bounds by an application of Perron's formula with the convexity bound $\zeta_K(\frac{1}{2}+it)\ll (|\Delta_K|(1+|t|)^{n})^{1/4+\eps}$. Hulse and Murty \cite{HulseMurty} used a variant of the convexity argument to give as an error term $O_{n,\eps}(|\Delta_K|^{\eps}x^{1-\eps/2})$, which is better in terms of the discriminant, but has unspecified constants in the degree $n$, and is suboptimal in terms of $x$. 
For general number fields $K$, the best result in terms of $x$ is due to Lao \cite{Lao}, who uses the subconvexity bound $\zeta_K(\frac{1}{2}+it)\ll_K (1+|t|)^{n/6+\eps}$(which follows from Kaufmans work \cite{Kaufman}) to produce an estimate with errorterm $O_K(x^{1-\frac{3}{n+6}+\eps})$. Weaker improvements on the exponent of $x$ which are not explicit on the constants can be found in \cite{Bordel2} and \cite{Nowak}. 
This corollary generalises the result for quadratic fields proven by Schmidt \cite{Schmidt}, namely that \[\sum_{k\leq x}a_k = \kappa_K x + O\left((\textup{Reg}_Kh_K)^{1/2}(\max(1,\log\textup{Reg}_Kh_K))^{1/2} x^{1/2}\right),\] with some absolute implied constant. Widmer \cite{Widmer} proved a comparable result for number fields of fixed degree, working however in a more general setting, counting points of bounded height in projective space over $K$. His work would imply an error term $O_n(\textup{Reg}_Kh_K x^{1-\frac{1}{n}})$ with constant depending only on the degree. That same error term would be implied in the PhD thesis of Gao \cite{Gao}. 

\section{Notation}
Let $K$ be a number field of discriminant $\Delta_K$ and of degree $n$, with $r_1$ real embeddings and $r_2$ complex embeddings, so that $n=r_1+2r_2$. Let $\sigma_1,\dots,\sigma_{r_1}$ be the real embeddings, and let $(\sigma_{r_1+1},\sigma_{r_1+r_2+1}),\dots,(\sigma_{r_1+r_2},\sigma_{r_1+2r_2})$ be the pairs of complex embeddings. Choose a system of $r=r_1+r_2-1$ fundamental units in the ring of integers $\OO_K$, denoted by  $\eps_1,\dots, \eps_r$. Dependent on this choice for a system of fundamental units, we consider the group of non-torsion units $U_K=\langle \eps_1,\dots \eps_r,\rangle$, and we define the positive constants\begin{equation}\label{defmax}m_j = \lceil\max_i|\sigma_i(\eps_j)|\rceil.\end{equation}

Let $\textup{Reg}_K$ be the regulator, $h_K$ the class number, and $\kappa_K$ the residue of the Dedekind-zeta function $\zeta_K(s)$ at $s=1$.

We will reserve the dummy variables $i$ and $j$ for specific uses, to make the notation more transparent. The embeddings will always be indexed by $i$, whereas the units will be indexed by $j$.

Recall the Minkowski embedding \begin{align*}\phi: K&\rightarrow \R^{r_1}\times \CC^{r_1}\\
\alpha&\mapsto (\sigma_1(\alpha), \dots,\sigma_{r_1+r_2}(\alpha)).\end{align*}
We will identify the image space $\R^{r_1}\times \CC^{r_2}$, which is called Minkowski space, with $\R^n$, by taking real and imaginary parts in the last $r_2$ coordinates. We will denote by $\|.\|$ the euclidean norm, where $|.|$ is the usual absolute value on $\R$ or $\CC$. Whenever we write a product of two vectors, componentwise multiplication is what is meant. Define $e_i=1$ for $i=1,\dots,r_1$ and $e_i=2$ for $i=r_1+1,\dots,r_1+r_2$, and define for each $x\in \R^{r_1}\times \CC^{r_2}$, the norm function $N(x)=\prod_{i=1}^{r_1+r_2} |x_i|^{e_i}$, which corresponds to the absolute value of the norm on $K$.

Given a choice of signs $\vec{\eta}=(\eta_1,\dots,\eta_{r_1})\in\{\pm 1\}^{r_1}$, define the open orthant $\R^{r_1}_{\vec{\eta}}=\{(x_1,\dots,x_{r_1})\in\R^{r_1}\mid \sgn(x_i)=\eta_i\}$.

\section{Counting lattice points}
We will count integral elements by embedding them in the Minkowski space as a (translate of a) lattice. We recall some definitions. Let $S$ be a subset of $\R^n$, where we assume that $n\geq 2$. We say that $S$ is of Lipschitz class $\cal{L}(n,M,L)$ if there are $M$ maps $\phi_1,\dots,\phi_M : [0,1]^{n-1}\longrightarrow \R^{n}$ such that $S$ is covered by the images of the maps $\phi_i$, and the maps satisfy the Lipschitz condition\begin{equation}\label{eq:lipscheq}\lVert\phi_i(x)-\phi_i(y)\rVert\leq L \lVert x-y\rVert \textup{ for } x,y\in [0,1]^{n-1}, i=1,\dots,M.\end{equation} 
We define the successive minima $\lambda_i, i=1,\dots,n$ of a lattice $\Gamma$ as \[\lambda_i = \inf\{\lambda\in\R \mid B(0,\lambda)\cap \Gamma \textup{ contains $i$ linearly independent vectors}\}.\]
We define the orthogonality defect $\Omega$ of a lattice $\Gamma$ as  \[\Omega = \inf_{(v_1, \dots , v_n)}\f{|v_1|\dots|v_n|}{\det \Gamma},\]
where the infimum runs over all bases $(v_1,\dots , v_n)$ of $\Gamma$.

In order to count lattice points, we will use the following theorem by \mbox{Widmer} \cite[Theorem 5.4]{Widmer}. The virtue of this theorem is that it is optimal in the successive minima and explicit in terms of the Lipschitz constant $L$ and the dimension $n$. 
\bl\label{Widmer}(Widmer) Let $\Gamma$ be a lattice in $\R^n$ with successive minima $\lambda_1,\dots,\lambda_n$ and orthogonality defect $\Omega$. Let $S$ be a bounded set in $\R^n$ such that the boundary $\partial S$ is of Lipschitz class $\cal{L}(n,M,L)$. Then $S$ is measurable, and moreover,
\[\left| |S\cap \Gamma| -\f{\textup{Vol}(S)}{\textup{det }\Gamma}\right|\leq M (2\sqrt{n}\Omega+4)^n \max_{0\leq i < n} \f{L^i}{\lambda_1\dots \lambda_i},\]
or, since unconditionally we have $\displaystyle\Omega\leq \f{n^{\frac{3}{2}n}}{(2\pi)^{\frac{n}{2}}}$,
\[\left| |S\cap \Gamma| -\f{\textup{Vol}(S)}{\textup{det }\Gamma}\right|\leq M n^{3n^2/2}\max_{0\leq i < n} \f{L^i}{\lambda_1\dots \lambda_i},\]
\el
\begin{rem} We point out that the same result holds for translates of lattices. Indeed, $|S\cap (\Gamma + v)| = |(S-v)\cap \Gamma|$, and $S-v$ is of the same Lipschitz class as $S$, and obviously has the same volume.


\end{rem}

\subsection{Fundamental domain}
We now recall the fundamental domain under the action of multiplication by the fundamental units. For full details we refer the reader to \cite{Lang}.
Let 
\begin{align*} f:\R^{r_1}\times \CC^{r_2} &\rightarrow \R^{r+1} \\
(x_1,\dots,x_{r_1+r_2}) &\mapsto (\xi,\xi_1,\dots,\xi_r),
\end{align*}
where $\xi$ and the $\xi_j$ are defined by \begin{equation}\label{eq:fundomeq} \log(|x_i|) = \f{\log(\xi)}{n} + \sum_{j=1}^{r} \xi_j \log(|\sigma_i(\eps_j)|),\quad\quad\quad i=1,\dots {r_1+r_2},
\end{equation}
or in other words, \begin{equation}\label{eq:fundomexpeq} |x_i|=\xi^{1/n}\prod_{j=1}^{r}|\sigma_i(\eps_j)|^{\xi_j},\quad\quad\quad i=1,\dots {r_1+r_2}.
\end{equation}
Note that since the units have norm one, it follows that $\xi=|N(x)|$. 

The fundamental domain is then defined as $\F=f^{-1}\left(\R_{+}\times [0, 1)^{r}\right)$. We remark that in what follows, we will always assume that $K\neq \Q$, since the results are trivial in this case. In the case that $K$ is an imaginary quadratic field, $r=0$, and in order for the proofs to remain valid, one should interpret the ``empty product" $\prod_{j=1}^r m_j$, as well as the ``empty determinant" Reg$_K$, as having the value $1$.

We summarise a number of well-known facts in the following proposition
\begin{prop} Consider the action of the non-torsion unit group $U_K$ on Minkowski space $\R^{r_1}\times\CC^{r_2}$ by multiplication. The set $\F$ is a fundamental domain for this action. It is a cone, and defining \[\F(a_1,b_1,\dots,a_r,b_r;X) = f^{-1}\left((0,X]\times \prod_{j=1}^r[a_j, b_j)\right),\]
we have that \begin{align}\vol(\F(a_1,b_1,\dots,a_r,b_r;X))&=X\prod_{j=1}^r(b_j-a_j)\vol(\F(0,1,\dots,0,1;1))\nonumber\\ &= X\prod_{j=1}^r(b_j-a_j)2^{r_1}\pi^{r_2}\reg_K.\label{eq:volumeeq}\end{align}
\end{prop}

We will denote some refinements of the fundamental domain in the following way.
\begin{align*}
&\F_\half(a_1,b_1,\dots,a_r, b_r; X) = f^{-1}\left((\half X,X]\times \prod_{j=1}^{r}[a_j, b_j)\right),\\ 
&\F_{\vec{\eta}}(a_1,b_1,\dots,a_r, b_r; X)=\F(a_1,b_1,\dots,a_r, b_r; X)\cap \left(\R^{r_1}_{\vec{\eta}}\times\CC^{r_2}\right),\\
&\F_{\half,\vec{\eta}}(a_1,b_1,\dots,a_r, b_r; X)=\F_\half(a_1,b_1,\dots,a_r, b_r; X)\cap \left(\R^{r_1}_{\vec{\eta}}\times\CC^{r_2}\right).\\
\end{align*}
Finally, $\F(X) = \F(0,1,\dots,0, 1; X)$, and similarly for when the subscripts $\half$ or $\vec{\eta}$ are present.

\subsection{A Partition of the fundamental domain.}

We will be counting points in the fundamental domain $\F$ using Lemma \ref{Widmer}, with an explicit error term in terms of the Lipschitz constant. In order to optimize this process, we will instead count points in a {\em smaller} variant of the fundamental domain with a more regular shape. This trick was first used by Schmidt in \cite{Schmidt} in the context of real quadratic fields, but has been applied in more general settings by Widmer \cite{Widmer}, Gao \cite{Gao}. The reason why we go through this partition step is because the Lipschitz constant of $\F_{\half, \vec{\eta}}(0,\frac{1}{m_1},\dots,0,\frac{1}{m_r};X)$ is essentially independent of $K$, whereas the Lipschitz constant of $\F_{\half, \vec{\eta}}(X)$ is exponential in the $m_j$. Schmidt's partition trick will thus immediately reduce the dependence on the size of the logarithms of the fundamental units from an exponential dependence to a polynomial dependence. 

\bl\label{schmidtpartition}(Schmidt's partition trick) Let $\Gamma$ be a set of points in $\R^{r_1}\times \CC^{r_2}$. Let $\vec{k}\in \prod_{j=1}^{r}([0,m_j)\cap\Z)$ and denote $\beta_{\vec{k}} = \big(\prod_{j=1}^{r}|\sigma_i(\eps_j)|^{-k_j/m_j}\big)\in{\R^{r_1}\times\CC^{r_2}}$, so that $N(\beta_{\vec{k}})=1$. Then 
\[|\Gamma \cap \F_{\half, \vec{\eta}}(X)|=\sum_{\vec{k}}|(\Gamma \cdot \beta_{\vec{k}}) \cap \F_{\half, \vec{\eta}}(0,\frac{1}{m_1},\dots,0,\frac{1}{m_r};X)|,\]
where the sum runs over $\vec{k}\in \prod_{j=1}^{r}([0,m_j)\cap\Z)$.
\el
\bp 
We split up the unit hypercube $[0,1)^{r}$ as a union of sets $\prod_{j=1}^{r}[\frac{k_j}{m_j}, \frac{k_j+1}{m_j})$, for $0\leq k_j\leq m_j-1, j=1,\dots, r$, so that 
\[ |\Gamma\cap\F_{\half, \vec{\eta}}(X)| = \sum_{\substack{(k_1,\dots,k_r)\\0\leq k_j\leq m_j-1}}|\Gamma\cap\F_{\half, \vec{\eta}}(\frac{k_1}{m_1}, \frac{k_1+1}{m_1}, \dots, \frac{k_{r}}{m_{r}}, \frac{k_{r}+1}{m_{r}}; X)|. \]
In each summand we multiply the elements in the intersection componentwise with \[\beta_{\vec{k}}=(\prod_{j=1}^r |\sigma_1(\eps_j)|^{-k_j/m_j},\dots,\prod_{j=1}^r |\sigma_{r_1+r_2}(\eps_j)|^{-k_j/m_j}),\] a norm one element since $\prod_{i=1}^{r_1+r_2}|\sigma_i(\eps_j)|^{e_i} = 1$. This gives a concrete bijection between the sets
\[\Gamma \cap \F_{\half, \vec{\eta}}(\frac{k_1}{m_1}, \frac{k_1+1}{m_1}, \dots, \frac{k_r}{m_r}, \frac{k_r+1}{m_r} ; X)\] and \[(\Gamma \cdot \beta_{\vec{k}}) \cap \F_{\half, \vec{\eta}}(0,\frac{1}{m_1},\dots,0,\frac{1}{m_r};X). \]
\end{proof}

The following lemma will be practical in computing the Lipschitz constant of the fundamental domain.
\bl Let $f:\R^m\rightarrow \R$ be a function such that $f(y_1,y_2,\dots,y_m)=c\prod_{j=1}^m g_j(y_j)$, where the $g_j:\R\rightarrow \R$ are functions that satisfy $|g_j(y_j) - g_j(y_j')|\leq K_j|y_j-y_j'|$ and $|g_j(y_j)|\leq M_j$. Then we have that \begin{equation}\label{eq:lemma1}\left|f(y_1,y_2,\dots,y_m)-f(y_1',y_2',\dots,y_m')\right|\leq c\sum_{j=1}^m K_j\prod_{k\neq j}M_k \lVert y-y'\rVert.\end{equation}
Let $h:\R^m\rightarrow \R^m, h=(h_1,\dots,h_m)$, where the $h_i:\R^m\rightarrow \R$ are functions that satisfy $|h_i(y)-h_i(y')|\leq L_i \lVert y-y'\rVert$. Then we have that \begin{equation}\label{eq:lemma2} \lVert h(y)-h(y')\rVert \leq \sqrt{m}\max_i L_i \lVert y-y'\rVert\end{equation}
\el
\bp
For the first inequality, let $y^i=(y_1',\dots,y_i',y_{i+1}\dots ,y_m)$. The result follows by writing the left hand side as $|f(y)-f(y^1)+f(y^1)-f(y^2)+\dots - f(y^m)|$, and the fact that $|y_i-y_i'|\leq \lVert y-y'\rVert$. The second inequality is obvious.
\end{proof}

\bl
The set $\partial\F_{\half, \vec{\eta}}(0,\frac{1}{m_1},\dots,0,\frac{1}{m_r};t^n)$ is of Lipschitz-class $\cal{L}(n,{2r+2},Lt),$ where \[L = (r+2\pi)^{3/2} e^r.\]
\el
\bp 
We note that the Lipschitz constant of a blown-up set $tR$ equals $tL$, where $L$ is the Lipschitz constant of $R$, hence we can restrict attention to the case $t=1$.

Note that for $(x_1,\dots,x_{r_1+r_2})\in\F_\half(0,\frac{1}{m_1},\dots,0,\frac{1}{m_r};1)$, by \eqref{eq:fundomeq}, $\log(|x_i|)$ is bounded from below by some constant depending only on $K$. Hence the $|x_i|$ are bounded away from zero, so that intersecting $\F_\half(0,\frac{1}{m_1},\dots,0,\frac{1}{m_r};1)$ with $\R^{r_1}_{\vec{\eta}}\times\CC^{r_2}$ does not create any additional boundary, that is,
\[\partial \cal{F}_{\half,\vec{\eta}}(0,\frac{1}{m_1},\dots,0,\frac{1}{m_r};1) = \partial\cal{F}_{\half}(0,\frac{1}{m_1},\dots,0,\frac{1}{m_r};1)\cap \left(\R^{r_1}_{\vec{\eta}}\times\CC^{r_2}\right).\]

In order to cover $\partial \cal{F}_{\half,\vec{\eta}}(0,\frac{1}{m_1},\dots,0,\frac{1}{m_r};1)$ with the images of maps from $[0,1]^{n-1}$, we consider the injective map \begin{align*}\tilde{f}: \R^{r_1}_{\vec{\eta}}\times\CC^{r_2} \quad \quad &\rightarrow \R^{n}\\ (x_1,\dots,x_{r_1+r_2})&\mapsto (f(x_1,\dots,x_{r_1+r_2}), \arg(x_{r_2+1}),\dots,\arg(x_{r_1+r_2})).\end{align*} 
Then 
\begin{align*}\partial\cal{F}_{\half}(0,\frac{1}{m_1},\dots,0,\frac{1}{m_r};1)\,\cap &\left(\R^{r_1}_{\vec{\eta}}\times\CC^{r_2}\right) \\
&= 
\partial f^{-1}\Big((\half ,1]\times \prod_{j=1}^r [0, \frac{1}{m_j})\Big)\cap \left(\R^{r_1}_{\vec{\eta}}\times\CC^{r_2}\right)\\
&= \partial \tilde{f}^{-1}\Big((\half ,1]\times \prod_{j=1}^r [0, \frac{1}{m_j})\times [0,2\pi)^{r_2}\Big),
\end{align*}

and since $\tilde{f}$ is continuous, we may conclude that \[\partial \cal{F}_{\half,\vec{\eta}}(0,\frac{1}{m_1},\dots,0,\frac{1}{m_r};1)=\tilde{f}^{-1}\partial\Big((\half ,1]\times \prod_{j=1}^r [0, \frac{1}{m_j})\times [0,2\pi)^{r_2}\Big),\]
so that $\partial \cal{F}_{\half,\vec{\eta}}(0,\frac{1}{m_1},\dots,0,\frac{1}{m_r};1)$ is covered by the images of the $2r+2$ maps $\tilde{f}^{-1}$, where one of the first $r+1$ coordinates is fixed to be either $\half$ or 1, or $0$ or $\frac{1}{m_j}$. Let us begin with the map fixing the first coordinate to be $\half$:
\begin{align*}\phi_0: (y_1,\dots,y_{n-1}) \longmapsto &\tilde{f}^{-1}(\half, \frac{y_1}{m_1},\dots, \frac{y_r}{m_r}, 2\pi y_{r+1}, \dots, 2\pi y_{n-1})\\
&=(x_1,\dots x_{r_1+ r_2}),
\end{align*}
where \[\begin{cases}\displaystyle x_i = \eta_i (\half)^{1/n}\prod_{j=1}^{r}|\sigma_i(\eps_j)|^{\frac{y_j}{m_j}}& \textup{if $i\leq r_1$},\\\displaystyle x_i = e^{2\pi i y_i}(\half)^{1/n}\prod_{j=1}^{r}|\sigma_i(\eps_j)|^{\frac{y_j}{m_j}}& \textup{if $r_1< i\leq r_1+r_2$}.
 \end{cases}\]

We will bound the Lipschitz constant of this map, using $\eqref{eq:lemma2}$, by $\sqrt{r_1+r_2} \max_i L_i$, where $L_i$ is such that $|x_i(y) - x_i(y')|\leq L_i \|y-y'\|$. To bound $L_i$, we will use $\eqref{eq:lemma1}$. Note that \[\left||\sigma_i(\eps_j)|^{y_j/m_j} - |\sigma_i(\eps_j)|^{y_j'/m_j} \right|\leq e |y_j-y_j'|, \quad \quad \textup{and}\quad \quad |\sigma_i(\eps_j)|^{y_j/m_j}\leq e,\] 
from which it follows that in the case $i\leq r_1$ we may take $K_j=M_j=e$ for $j\leq r$, $K_j=0, M_j=1$ for $j>r$. Hence, $L_i \leq (\half)^{1/n}re^r$. However, in the case $r_1< i\leq r_1+r_2$, we have $K_i=2\pi, M_i=1$, and all other $K_j, M_j$ are as before. In this case $L_i\leq (\half)^{1/n}(re+2\pi)e^{r-1}$.
Thus, the Lipschitz constant of this map $\phi_0$ is bounded by $(\half)^{1/n}\sqrt{r_1+r_2}(r+2\pi) e^r$. Analogously, the map which fixed the first coordinate to be $1$ instead of $\half$ has a Lipschitz constant which is bounded by $\sqrt{r_1+r_2}(r+2\pi) e^r$.

We use the same strategy to bound the Lipschitz constant of other maps $\phi_k$ which fix the $k$'th coefficient to be $a\in\{0,\frac{1}{m_k}\}$:

\begin{align*}\phi_k : (y_0,\dots,\widehat{y_k},\dots ,y_{n-1}) \longmapsto &\tilde{f}(\frac{1+y_0}{2}, \frac{y_1}{m_1},\dots, a, \dots \frac{y_r}{m_r}, 2\pi y_{r+1}, \dots, 2\pi y_{n-1})\\
&=(x_1,\dots , x_{r_1+ r_2}),
\end{align*}
where \[\begin{cases}\displaystyle x_i = \eta_i \,\left(\frac{1+y_0}{2}\right)^{1/n}\,\prod_{\substack{j=1\\j\neq k}}^{r}|\sigma_i(\eps_j)|^{\frac{y_j}{m_j}} |\sigma_i(\eps_k)|^{a}& \textup{if $i\leq r_1$},\\\displaystyle x_i = e^{2\pi i y_i}\left(\frac{1+y_0}{2}\right)^{1/n}\prod_{\substack{j=1\\j\neq k}}^{r}|\sigma_i(\eps_j)|^{\frac{y_j}{m_j}}|\sigma_i(\eps_k)|^{a}& \textup{if $r_1< i\leq r_1+r_2$}.
 \end{cases}\]

We will bound the Lipschitz constant of this map, using $\eqref{eq:lemma2}$, by $\sqrt{r_1+r_2} \max_i L_i$, where $L_i$ is such that $|x_i(y) - x_i(y')|\leq L_i \|y-y'\|$. To bound $L_i$, we will use $\eqref{eq:lemma1}$. As with $\phi_0$, in the case $i\leq r_1$ we may take $K_j=M_j=e$ for $1\leq j\leq r$, $K_j=0, M_j=1$ for $j>r$. Since \[\left|\left(\frac{1+y_0}{2}\right)^{1/n} - \left(\frac{1+y_0'}{2}\right)^{1/n}\right|\leq \f{1}{2^{1/n}n}|y_0-y_0'|,\] we may take $K_0=1/n$ and $M_0=1$. Hence, $L_i \leq  |\sigma_i(\eps_k)|^{a}(r-1+1/n)e^{r-1}\leq re^r$. However, if $r_1< i\leq r_1+r_2$, we note that $K_i=2\pi, M_i=1$, and all other $K_j, M_j$ are as before. In this case $L_i\leq (r+2\pi)e^{r}$.
Thus, the Lipschitz constant of this map $\phi_k$ is bounded by $\sqrt{r_1+r_2}(r+2\pi) e^r$.
\end{proof}

We will need the following three lemmata.
\bl\label{inverseamgm} Let $x=(x_1,\dots ,x_{r_1+r_2}) \in \F(0, \frac{1}{m_1}, \dots, 0, \frac{1}{m_r}; X)$. Then \[ \|x\| \leq \sqrt{r+1}\,e^{r}|N(x)|^{1/n} \]
\el
\bp 
This follows immediately from \eqref{eq:fundomexpeq}, the fact that $\log|\sigma_i(\eps_j)|\leq m_j$ and the inequality between the euclidean norm and the max-norm.
\end{proof}
We now fix our choice of fundamental units such that the product of the constants $m_j$ approximates the regulator of the number field. This is possible by the following lemma.
\bl\label{Regvsproduct}
There is a choice of fundamental units $\eps_1,\dots,\eps_r$ such that \[\frac{1}{2^{r+1}(r+1)^{r/2}}\textup{Reg}_K\leq \prod_{j=1}^{r}m_j \leq (2rn)^{2r}\textup{Reg}_K.\]
\el
\bp 
Consider the logarithmic unit lattice $\Lambda=\l(\phi(U_K))\in\R^{r+1}$ with $\det\Lambda = \sqrt{r+1}\textup{Reg}_K$, where $\l(x_1,\dots,x_{r_1+r_2})=(e_1\log|x_1|,\dots,e_{r_1+r_2}\log|x_{r_1+r_2}|)$. Since units have norm one, this lattice lies in the $r$-dimensional hyperplane $x_1+\dots+x_{r_1+r_2}=0$.
The unconditional bound on the orthogonality defect implies that we can find a basis $v_1,\dots,v_r$ of $\Lambda$, such that \begin{equation}\label{eq:latticebasiseq}1\leq \frac{\| v_1\|\dots\| v_r\|}{\det\Lambda}\leq \f{r^{\frac{3}{2}r}}{(2\pi)^{\frac{r}{2}}}.\end{equation} Now choose $\eps_j$ such that $\ell(\phi(\eps_j))=v_j$. Since the $v_j$ form a basis, the $\eps_j$ form a set of fundamental units. 
For the lower bound, note that \[m_j=\lceil \max_i(\log|\sigma_i(\eps_j)|)  \rceil \geq \frac{1}{2} \max_i(e_i\log|\sigma_i(\eps_j)|) \geq \frac{1}{2\sqrt{r+1}} \|v_j\|,\]
and combine this with \eqref{eq:latticebasiseq}. For the upper bound, the case $n=2$ can be checked in an elementary way. For $n\geq 3$, note that by a theorem of Dobrowolski \cite{Dobro}, we have that $\max_i |\sigma_i(\eps_j)| \geq 1+\frac{\log n}{6n^2}$, and hence $\max_i \log|\sigma_i(\eps_j)| \geq \frac{1}{7n^2}$. This implies that \[m_j\leq 7n^2\max_i(e_i\log|\sigma_i(\eps_j)|) \leq 7n^2 \|v_j\|,\]
which gives the upper bound when combined with \eqref{eq:latticebasiseq}.
\end{proof}

\bl\label{integraleval} \[\int_{\R^r} \frac{dx_1\dots dx_r}{\max_{1\leq i\leq r+1}\big(\prod_{j=1}^{r}|\sigma_i(\eps_j)|^{x_j/m_j}|\sigma_i(\alpha)|\big)^{n-1}} = \frac{1}{|N(\alpha)|^{\frac{n-1}{n}}} \frac{m_1\dots m_r}{\textup{Reg}_K} \left(\frac{n}{n-1}\right)^r\]
\el
\bp
We start with the coordinate transformation $x_j \mapsto x_j/m_j$. Then, introducing the coordinate transformation $ y_i = \sum_{j=1}^r \log |\sigma_i(\eps_j)| x_j$, $i=1,\dots , r$, with determinant $\prod_{i=1}^r e_i \frac{1}{\textup{Reg}_K}$, the integral  equals \begin{align*}\prod_{i=1}^r e_i&\frac{m_1\dots m_r}{\textup{Reg}_K}\int_{\R^r} \frac{dy_1\dots dy_r}{\max(e^{y_1}|\sigma_1(\alpha)|,\dots, e^{y_r}|\sigma_r(\alpha)|, e^{-Y}|\sigma_{r+1}(\alpha)|)^{n-1}},\end{align*}
where $Y=\frac{1}{e_{r+1}}\sum_{j=1}^{r}e_jy_j$. We now apply the translation $y_i\mapsto y_i +\log|\sigma_i(\alpha)| - \frac{\log|N(\alpha)|}{n}$ which allows us to pull out the factor $|N(\alpha)|^{\frac{n-1}{n}}$, and the integral now equals
\begin{align*}
 \prod_{i=1}^r e_i \frac{m_1\dots m_r}{\textup{Reg}_K|N(\alpha)|^{\frac{n-1}{n}}}\int_{\R^r} \frac{dy_1\dots dy_r}{\max(e^{y_1},\dots, e^{y_r}, e^{-Y})^{n-1}}
.\end{align*}
Next, we introduce new coordinates $z_1,\dots , z_r$ such that $(n-1)y_i = nz_i - Z$, where $Z = \sum_{i=1}^r e_iz_i$. We compute the determinant of this coordinate transformation, using the matrix determinant Lemma: \begin{align*}\frac{1}{(n-1)^r}&\det\left(nI - (e_1,\dots,e_r)^T(1,\dots,1)\right)\\
& = \frac{1}{(n-1)^r}\det(nI) \left(1-(1,\dots,1)(\frac{1}{n}I)(e_1,\dots,e_r)^T\right)\\
& =  e_r \frac{n^{r-1}}{(n-1)^r}.\end{align*}
Noting that \[(n-1)Y = \frac{n-1}{e_{r+1}}\sum_{i=1}^{r} e_iy_i =  \frac{n}{e_{r+1}}Z- \frac{1}{e_{r+1}}\sum_{i=1}^r e_iZ = Z, \]
the integral equals \[ \prod_{i=1}^{r+1} e_i \frac{n^{r-1}}{(n-1)^r} \frac{m_1\dots m_r}{\textup{Reg}_K|N(\alpha)|^{\frac{n-1}{n}}}I,\]
where \[I=\int_{\R^r} \frac{dz_1\dots dz_r}{\max(e^{nz_1 - Z},\dots, e^{nz_r - Z}, e^{ - Z})}.\]
Now let $R_k = \{(z_1,\dots, z_r)\in\R^r\mid z_k \geq \max(0, z_1, \dots, z_r)\}$. In this region, the integral equals \begin{align*}\int_{R_k}\frac{dz_1\dots dz_r}{e^{nz_k-Z}} &= \int_{z_k\geq0} \frac{dz_k}{e^{(n-e_k)z_k}}\prod_{\substack{i=1\\i\neq k }}^r\int_{z_i\leq z_k}\frac{dz_i}{e^{-e_iz_i}}\\
&= \int_{z_k\geq0} \frac{dz_k}{e^{(n-e_k)z_k}}\prod_{\substack{i=1\\i\neq k }}^r\frac{1}{e^{-e_iz_k} e_i} \\
&= \int_{0}^\infty \frac{dz_k}{e^{e_{r+1}z_k}} \prod_{\substack{i=1\\i\neq k }}^r\frac{1}{e_i} = \prod_{\substack{i=1\\i\neq k }}^{r+1}\frac{1}{e_i}
,\end{align*}
while in the region $S= \{(z_1,\dots,z_r)\in \R^r\mid z_i<0,\,\, \forall i \}$, the integral equals \begin{align*} \int_{z_i<0\, \forall i} \frac{dz_1\dots dz_r}{e^{-Z}} = \prod_{i=1}^r \int_0^\infty \frac{dz_i}{e^{e_iz_i}} = \prod_{i=1}^{r} \frac{1}{e_i}.\end{align*}
Now $\R^r$ is the disjoint union of the $R_k$ and the set $S$, up to sets of measure zero, hence, the full integral $I$ equals $\frac{\sum_{i=1}^{r+1}e_i}{\prod_{i=1}^{r+1}e_i} = \frac{n}{\prod_{i=1}^{r+1} e_i}$, and the result follows.
\end{proof}

\subsection{The counting theorem}
We are now ready to prove our key result counting elements of bounded norm in congruence classes modulo ideals in rings of integers, with given signs under the real embeddings, up to multiplication by units of infinite order. Let $\frak{C}$ be an ideal class, and let $\b_1, \b_2, \dots$ be the integral ideals in $\frak{C}$ ordered such that $N(\b_1)\leq N(\b_2)\leq \dots$. Define \[\frak{N}(\frak{C})=\sum_{i=1}^{m_1\dots m_r}\frac{1}{N(\b_i)^{\frac{n-1}{n}}}.\]

\bt\label{keylemma}Let $\frak{a}, \frak{f}$ be coprime integral ideals of the ring of integers $\cal{O}_K$ of a number field $K$ of degree $n$. Let $f\in \O_K$, let $\vec{\eta} \in\{\pm\}^{r_1}$, and let $\frak{C}$ be the ideal class of $\frak{a}\frak{f}$. Then for all $t\geq0$,
\begin{align*}\left| \left\{ \alpha\in \a \,
\def\arraystretch{0.75}
\begin{array}{|l} \phi(\alpha)\in \cal{F}_{\vec{\eta}}(t^{n}) \\ \alpha \equiv f\bmod \frak{f} \end{array}\right\}\right| &= \f{(2\pi)^{r_2}\textup{Reg}_K}{\sqrt{|\Delta_K|}} \frac{t^{n}}{N(\a\frak{f})} \\
&\quad+O^{*}\left(n^{\frac{3}{2}n^{2}}e^{4n^2}\frak{N}(\frak{C}^{-1})\frac{t^{n-1}}{N(\frak{a}\frak{f})^{\f{n-1}{n}}}+m_1\dots m_r\right),\end{align*}
where the term $m_1\dots m_r$ can be dropped if $\frak{f}=\O_K$.
\et
\bp
We shall prove that \begin{align}\label{eq:dyadiceq} & \left| \left\{ \alpha\in \a \,
\def\arraystretch{0.75}
\begin{array}{|l} \phi(\alpha)\in \cal{F}_{\half,\vec{\eta}}(t^{n}) \\ \alpha \equiv f\bmod \frak{f} \end{array}\right\}\right| = \f{(2\pi)^{r_2}\textup{Reg}_K}{\sqrt{|\Delta_K|}N(\a\frak{f})} \frac{t^{n}}{2} \\ &\quad\quad\quad\quad\quad\quad\quad\quad\quad\quad\quad\quad\quad+O^{*}\left(\frac{n^{\frac{3}{2}n^{2}}e^{4n^2}}{2}\frak{N}(\frak{C}^{-1})\frac{t^{n-1}}{N(\frak{a}\frak{f})^{\f{n-1}{n}}}+E(t)\right),\nonumber\end{align}
such that the terms $E(t)$ are not present if $\frak{f}=\O_K$, and such that \[\sum_{m=0}^\infty E(t/2^{m/n})\leq m_1\dots m_r.\] The statement in the theorem then follows by a dyadic summation.

Since $\a$ and $\frak{f}$ are coprime, we may choose $f$ to be an element of $\a$ without changing its residue class mod $\frak{f}$. The left hand side then equals \[\left|\left(\phi(\a\frak{f})+\phi(f)\right)\cap \F_{\half,\vec{\eta}}(t^n)\right|.\] We use Schmidt's partition trick, Lemma \ref{schmidtpartition}, to write the left hand side as \begin{equation}\label{eq:summandeq}\sum_{\vec{k}}\left|(\phi(\a\frak{f}) + \phi(f))\cdot \beta_{\vec{k}} \bigcap \F_{\half,\vec{\eta}}(0,\frac{1}{m_1},\dots ,0,\frac{1}{m_r} ;t^n)\right|.\end{equation}
We estimate each summand using Lemma \ref{Widmer}, since by the remark after the theorem, the estimate holds whether or not the translation by $\phi(f)\cdot \beta_{\vec{k}}$ is present. We consider the main term first, \[\frac{\textup{Vol}(\F_{\half,\vec{\eta}}(0,\frac{1}{m_1},\dots ,0,\frac{1}{m_r} ;t^n))}{\det \left(\phi(\a\frak{f})\cdot\beta_{\vec{k}}\right)}.\]
Since $\beta_{\vec{k}}$ is of norm one, the transformation is unimodular, and a standard computation gives \[\det \left(\phi(\a\frak{f})\cdot\beta_{\vec{k}}\right) = \det (\phi(\a\frak{f})) = 2^{-r_2}\sqrt{|\Delta_K|}N(\a\frak{f}).\] Using \eqref{eq:volumeeq} and the fact that $\F$ is symmetric under reflection through the coordinate hyperplanes in the first $r_1$ coordinates, we see that \[\textup{Vol}(\F_{\half,\vec{\eta}}(0,\frac{1}{m_1},\dots ,0,\frac{1}{m_r} ;t^n)) = \frac{1}{2^{r_1}}\frac{2^{r_1}\pi^{r_2}\textup{Reg}_K}{m_1\dots m_r}\frac{t^n}{2},\] so that the main term in each summand equals \[\frac{1}{m_1\dots m_r}\frac{(2\pi)^{r_2}\textup{Reg}_K}{\sqrt{|\Delta_K|}N(\a\frak{f})}\frac{t^n}{2},\] which yields the main term in \eqref{eq:dyadiceq} after summation over $\vec{k}\in\prod_{j=1}^r [0,m_j)\cap\Z$.

We now consider the error term. Since $\F_{\half,\vec{\eta}}(0,\frac{1}{m_1},\dots ,0,\frac{1}{m_r} ;t^n)$ is in Lipschitz class $\mathcal{L}(n,2r+2, Lt)$, where $L=(r+2\pi)^{3/2}e^r$, the error term in estimating each summand is bounded by \[(2r+2)n^{3n^2/2}\max_{0\leq i\leq n-1}\frac{L^it^i}{\lambda_1(\a\frak{f},\vec{k})\dots \lambda_i(\a\frak{f},\vec{k})}\leq (2r+2)n^{3n^2/2}\max_{0\leq i\leq n-1} \left(\frac{Lt}{\lambda_1(\a\frak{f},\vec{k})}\right)^i,\]
where $\lambda_i(\a\frak{f},\vec{k})$ are the successive minima of the lattice $\phi(\a\frak{f})\cdot \beta_{\vec{k}}$. 

We now claim that we may replace this maximum by setting $i=n-1$. This is certainly the case if $\lambda_1(\a,\vec{k})\leq Lt$. In case that $Lt<\lambda_1(\a,\vec{k})$, we have to be careful -- we will prove that the corresponding summand in $\eqref{eq:summandeq}$ is at most one, even after the dyadic summation. Consider two elements $x,y$ in the intersection of $(\phi(\a\frak{f}) + \phi(f))\cdot \beta_{\vec{k}}$ and $\F_{\vec{\eta}}(0,\frac{1}{m_1},\dots ,0,\frac{1}{m_r} ;t^n)$. By Lemma \ref{inverseamgm} we have that $\| x -y \|\leq \|x\| + \|y\| \leq 2\sqrt{r+1}e^r t$. However, $x-y \in \phi(\a\frak{f})\cdot\beta_{\vec{k}}$, and hence $\|x-y\|\geq \lambda_1(\a\frak{f},\vec{k})>Lt$, which is a contradiction. Hence there can only be one exceptional point for each summand in $\eqref{eq:summandeq}$, even after the dyadic summation. We note that in case that $\frak{f}=\O_K$, we can use the same argument for $x$ in place of $x-y$ to get a contradiction, so in that case there are no exceptional points. We write these potential contributions in a separate error term as $E(t)$, which after dyadic summation will sum up to at most $m_1\cdots m_k$. 

We will now show that the main term is bounded by the error term with the maximum replaced by setting $i=n-1$. More precisely, we need to show that \begin{equation}\label{eq:mainvserroreq}\frac{1}{m_1\dots m_r}\frac{(2\pi)^{r_2}\textup{Reg}_K}{\sqrt{|\Delta_K|}N(\a\frak{f})}\frac{t^n}{2}\leq (2r+2)n^{3n^2/2} \frac{L^{n-1}t^{n-1}}{\lambda_1(\a\frak{f},\vec{k})^{n-1}}.\end{equation}
We may multiply the right hand side with the factor $\frac{Lt}{\lambda_1(\a\frak{f},\vec{k})}< 1$. By Minkowski's second theorem we have that \[\lambda_1(\a\frak{f},\vec{k})^{n}\leq \frac{2^n}{\pi^{n/2}}\Gamma(1+\frac{n}{2}) \det \phi(\a\frak{f}).\beta_{\vec{k}} = \frac{2^n}{\pi^{n/2}}\Gamma(1+\frac{n}{2}) 2^{-r_2}\sqrt{|\Delta_K|}N(\a\frak{f}),\] and by Lemma \ref{Regvsproduct} we have that Reg$_K\leq 2^{r+1}(r+1)^{r/2} m_1\dots m_r$, so after plugging in these inequalities, \eqref{eq:mainvserroreq} follows from the inequality \[\frac{\pi^{r_2}}{2} \frac{2^n}{\pi^{n/2}}\Gamma(1+\frac{n}{2})  2^{r+1}(r+1)^{r/2} \leq (2r+2)n^{3n^2/2} L^n.\]
%
%
%
Hence, the total error term in \eqref{eq:dyadiceq} is bounded by \begin{equation}\label{eq:totalerroreq}(2r+2)n^{3n^2/2}L^{n-1}t^{n-1}\sum_{\vec{k}} \frac{1}{\lambda_1(\a\frak{f},\vec{k})^{n-1}}+E(t).\end{equation}
The last part of the proof is to bound the sum over $\vec{k}$. Define \[\mu(\a\frak{f},\vec{k})=\min_{\alpha\in\a\frak{f}} \max_i (|\sigma_i(\alpha)| \prod_{j=1}^r|\sigma_i(\eps_j)|^{-k_j/m_j}),\] and note that $\lambda_1(\a\frak{f},\vec{k})\geq \mu(\a\frak{f},\vec{k})$. For each element $\alpha$, define $K_\alpha$ as the set of $\vec{k}\in\prod_{j=1}^r [0,m_j)\cap\Z$ such that $\alpha$ attains the minimal value in the definition of $\mu(\a\frak{f},\vec{k})$. We change the order in the sum 
\begin{align*} \sum_{\vec{k}} \frac{1}{\lambda_1(\a\frak{f},\vec{k})^{n-1}} 
& \leq \sum_{\vec{k}} \frac{1}{\mu(\a\frak{f},\vec{k})^{n-1}}\\
& \leq \sum_{\alpha\in \a\frak{f}}{'} \sum_{u\in U_K} \sum_{\vec{k}\in K_{u\alpha}}\frac{1}{\max_i(|\sigma_i(u\alpha)| \prod_{j=1}^r|\sigma_i(\eps_j)|^{-k_j/m_j})^{n-1}}\\
& \leq \sum_{\alpha\in \a\frak{f}}{'} \sum_{\vec{k}\in \Z^r}\frac{1}{\max_i(|\sigma_i(\alpha)| \prod_{j=1}^r|\sigma_i(\eps_j)|^{-k_j/m_j})^{n-1}},
\end{align*}
where the dash signifies that the sum runs over pairwise non-associated elements $\alpha$, and the number of summands is at most $m_1\dots m_r$.

We denote by $g(\vec{k})$ the function in the last sum above. For any $\vec{\delta}\in [0,1]^{r}$, we see that $g(\vec{k})\leq (e^r)^{n-1} g(\vec{k}+\vec{\delta})$, and hence \[\sum_{k\in \Z^r} g(\vec{k})\leq e^{r(n-1)}\int_{\vec{k}\in\R^r} g(\vec{k})dk_1\dots dk_r.\]

By Lemma \ref{integraleval}, the sum is then bounded \[\sum_{\vec{k}} \frac{1}{\lambda_1(\a\frak{f},\vec{k})^{n-1}} \leq e^{rn}\frac{m_1\dots m_r}{\textup{Reg}_K}\sum_{\alpha\in \a\frak{f}}{'} \frac{1}{|N(\alpha)|^{\frac{n-1}{n}}}.\]

Since the elements $\alpha\in \a\frak{f}$, it follows that $(\alpha)=\a\frak{f}\b$, for a certain $\b\in \frak{C}^{-1}$. Since the $\alpha$ appearing in the sum are non-associated, all respective ideals $\b$ are different. Hence \[\sum_{\alpha\in \a\frak{f}}{'} \frac{1}{|N(\alpha)|^{\frac{n-1}{n}}}\leq \frac{1}{N(\a\frak{f})^{\frac{n-1}{n}}}\frak{N}(\frak{C}^{-1}).\]

Gathering all this into \eqref{eq:totalerroreq}, we see that the total error in \eqref{eq:dyadiceq} is bounded by \[(2r+2)n^{3n^2/2}L^{n-1} e^{rn} \frac{m_1\dots m_r}{\textup{Reg}_K}\frac{t^{n-1}}{N(\a\frak{f})^{\frac{n-1}{n}}}\frak{N}(\frak{C}^{-1})+E(t).\]
Using Lemma \ref{Regvsproduct} to bound $m_1\dots m_r \leq (2rn)^{2r}\textup{Reg}_K$, and plugging in $L=\sqrt{n}(r+2\pi)e^r$, the final step is to check that \[(2r+2)n^{3n^2/2}((r+2\pi)^{3/2}e^r)^{n-1} e^{rn} (2rn)^{2r}\leq n^{3n^{2}/2}\frac{e^{4n^2}}{2}.\]
\end{proof}
\begin{rem} We give some comments on the error term. In terms of $N(\a\frak{f})$, the result is best possible. The exponent of $t$ might be improved, but would require an entirely different approach, as it is a consequence of our use of the Lipschitz-constant. In any case, any improvement of the exponent of $t$ will naturally result to the same exponent in the successive minima $\lambda_i$, and thus in the power of $N(\frak{a}\frak{f})$, which for our purposes would make little difference in the end. With regards to the field constants, we will see in Lemma \ref{boundfinally} that $\frak{N}(\frak{C}^{-1})$ is essentially $\textup{Reg}_K^{1/n}$; this substantial improvement over the linear dependence in the regulator that one gets by Schmidt's partition trick is the result of our careful analysis of the successive minima. In terms of the degree $n$, the main factor $n^{3n^2/2}$ comes from Lemma \ref{Widmer}, and with our approach, it can only be improved by finding a bound on the orthogonality defect of ideal lattices. As far as the author is aware, there is no reason to assume that ideal lattices have small orthogonality defect.
\end{rem}
\begin{rem}
In the case $\a=\frak{f}=\cal{O}_K$, Theorem \ref{keylemma} is essentially (up to the quality of the constant in the error term) the same as the main result in \cite{MurtyVanOrder}, although we should mention that their proof is incomplete.
\end{rem}

In order to bound the quantity $\frak{N}(\frak{C})$, we will need the following lemma. We allow some superfluous logarithms (at least in the limit $y\rightarrow \infty$), for the sake of obtaining a clean explicit statement valid for all $y$.

\begin{lemma}\label{boundfinally} Let $\b_1,\b_2,\dots $ be the integral ideals in $\O_K$, ordered such that $N(\b_1)\leq N(\b_2) \leq \dots$. Then, for all $y\geq2$, \[\sum_{i=1}^{y}\frac{1}{N(\b_i)^{\frac{n-1}{n}}} \leq 6ny^{1/n}\log(y)^{\frac{(n-1)^2}{n}}.\]
\end{lemma}
\bp Let $m\in\N$, then $a_m$ can be majored by assuming that all prime factors of $m$ split completely. Then, \[a_m \leq d_n(m),\] where $d_n(m)$ is the number of ways you can write $m$ as a product of $n$ integers. A result by Bordell\`es \cite{Bordel} gives the explicit upper bound \[\sum_{1\leq m\leq x}d_n(m)\leq 2x(\log x)^{n-1},\]
if $x\geq6$. It follows that $N(\b_i)\geq \frac{i}{2(\log i)^{n-1}}=:x(i)$, if $x(i)$ is at least $6$. Indeed, if $N(\b_i)<x(i)$, then it would follow that $i\leq \sum_1^{x(i)}a_m < i$. For those $i\geq 2$ such that $x(i)\leq 6$, note that $N(\b_i)\geq 2$ and $2\geq x(i)/3$. Hence, for all $i\geq 2$, \[N(\b_i)\geq \frac{i}{6(\log i)^{n-1}}.\] The result follows immediately.
\end{proof}

We can now easily deduce the explicit bound on the number of ideals of norm up to $x$ stated in the introduction, hence making Landau's classical proof of the meromorphic continuation of $\zeta_K(s)$ to Re$(s)>1-\frac{1}{n}$ explicit. Recall that $\zeta_K(s) = \sum_{n=1}^{\infty}\frac{a_n}{n^s}$, where $a_n$ is the number of ideals of norm $n$, and $\kappa_K = \textup{res}_{s=1}\zeta_K(s)$. We repeat the statement here for convenience.

\begin{cornonumber*}Let $K$ be a number field of degree $n$, and let $a_k$ be the number of ideals of norm $k$. Then, for all $x\geq 1$, \[\sum_{k\leq x}a_k = \kappa_K x + O^{*}\left(n^{10n^2}(\textup{Reg}_Kh_K)^{1/n}(1+ \log\textup{Reg}_Kh_K)^{\frac{(n-1)^2}{n}} x^{1-\frac{1}{n}}\right).\]
Moreover, let $\mathfrak{c}$ be an element of the class group of $K$, and let $b_k$ be the number of ideals of norm $k$ which lie in $\mathfrak{c}$. Then,
\[\sum_{k\leq x}b_k = \frac{\kappa_K}{h_K} x + O^{*}\left(n^{10n^2}(\textup{Reg}_K)^{1/n}(1+ \log\textup{Reg}_K)^{\frac{(n-1)^2}{n}} x^{1-\frac{1}{n}}\right).\]
\end{cornonumber*}
\begin{proof}
Let $\b$ be an integral ideal in the inverse class $\frak{c}^{-1}$. We then have a bijection between the set of integral ideals in $\frak{f}$ and the set of principal ideals in $\b$ \[\a \mapsto \a\b=(\alpha),\] with inverse $(\alpha)\mapsto(\alpha)\b^{-1}$, which is an integral ideal since $\b|(\alpha)$. Thus we may count the ideals $\a$ in the class $\frC$ of norm up to $x$ by counting the principal ideals inside $\b$ of norm up to $xN(\b)$. We do this by counting elements up to multiplication by units. If $\omega$ is the number of roots of unity in $K$, we may write the number of all ideals in $\frak{c}$ of norm up to $x$ as \[\frac{1}{\omega}|\{\alpha\in\b\mid\phi(\alpha)\in\cal{F}(xN(\b))\}|.\]
Using the key lemma, we see that this equals \begin{equation}\label{eq:singleclass}\f{\kappa_K}{h_K} x + O^{*}\left(2^{r_1}n^{\frac{3}{2}n^{2}}e^{4n^2}\mathfrak{N}(\frak{c})x^{1-\f{1}{n}}\right).\end{equation}
Using Lemma \ref{boundfinally} with $y=m_1\dots m_r\leq (2rn)^{2r} \textup{Reg}_K$, the statement for a single class follows after checking that \[2^{r_1} n^{\frac{3}{2}n^2}e^{4n^2} 6n (2rn)^{2r/n}(2r\log(2rn))^{\frac{(n-1)^2}{n}}\leq n^{10n^2}.\]
Furthermore, summing \eqref{eq:singleclass} over all ideal classes $\frak{c}$, we get \[\sum_{m\leq x}a_m = \kappa x + O^{*}\left(2^{r_1}n^{\frac{3}{2}n^{2}}e^{4n^2}\sum_{\frak{c}}\mathfrak{N}(\frak{c})x^{1-\f{1}{n}}\right).\]
Using Lemma \ref{boundfinally} again, now with $y=h_Km_1\dots m_r$ then yields the result. \end{proof}

\section{The Selberg sieve in rings of integers}

We will use the Selberg sieve in the ring of integers $\O_K$, or more precisely, in the arithmetic semigroup of the integral ideals. We note that in the literature, the phrase ``Selberg sieve in rings of integers" has been used with another meaning (see \cite{Hinz}, \cite{Rieger}, \cite{Sarges} and \cite{Schaal}), namely the sifting of integral elements whose conjugates are restricted to lie in a box. This is a quite different story from our setup of working with ideals, which will amount to considering integral elements up to multiplication by units.

Concretely, let $\A$ be a set of nonzero integral ideals in $\O_K$, let $\P$ be a set of prime ideals in $\O_K$, and let $\overline{\P}$ be the complement of $\P$ in the set of prime ideals. The sifted set is defined, analogously to the case $K=\Q$, as \[S(\A,\P,z) = \{\a\in\A \;\mid (\a, \P(z))=1 \},\quad \textup{where}\quad\P(z)=\prod_{\substack{\p\in\P\\ N(\p)\leq z}}\p.\]
In what follows, $\mu$ will denote the M\"obius function on the integral ideals, $\nu$ is the function counting prime divisors, defined by $\nu(\p_1^{e_1}\dots\p_k^{e_k})=\sum e_i$, and $\phi(\d)=|(\O_K/\d)^{*}| = \prod_{\p^k\| \d}N(\p^{k-1})(N(\p)-1)$ is the generalisation of the euler totient. One may then completely transfer Chapter 3.1 in \cite{HalberstamRichert} to our situation, working with integral ideals instead of integers. In particular, the Selberg sieve weights are now a collection of reals $\lambda_\d$ where $\d$ ranges over the squarefree integral ideals. The only caveat is that all instances of integers in situations where its magnitude is the relevant property should be translated into the norms of the integral ideals. For completeness we give the definition of the relevant quantities. Let $\omega$ be a multiplicative function on the integral ideals.
\ba &g(\d) = \f{\omega(\d)}{N(\d)\prod_{\p|\d}(1-\f{\omega(\p)}{N(\p)})}\\
&G_{\mathfrak{k}}(x)=\sum_{\substack{N(\d)<x\\(\d,\mathfrak{k})=1}}\mu^2(\d)g(\d), \textup{ and }G(x)=G_{\O_K}(x)\\
&\lambda_\d = \f{\mu(\d)}{\prod_{\p|\d}(1-\f{\omega(\p)}{N(\p)})}\f{G_\d(z/N(\d))}{G(z)}
\end{align*}

With this slightly generalised setup, Theorem 3.2 in \cite{HalberstamRichert} immediately generalises to the following Theorem.
\bt\label{selberg} Assume that $|\A_\d| = \frac{\omega(\d)}{N(\d)}X + R_\d$, where $\omega$ is a multiplicative function on the integral ideals in $\O_K$ satisfying $0\leq \frac{\omega(\p)}{N(\p)}\leq 1-\frac{1}{A}$ for some constant $A$. Then
\[S(\A,\P,z)\leq \frac{X}{G(z)} + \Sigma_2,\]
where \[\Sigma_2 \leq \sum_{\substack{N(\d_i)\leq z\\\d_i|\P(z)\\i=1,2}} |R_{[\d_1,\d_2]}|\leq \sum_{\substack{N(\d)<z^2\\\d|\P(z)}}3^{\nu(\d)}|R_{\d}|.\]
\et

In the case that $\omega= 1$, we can give a lower bound for $G(z)$.

\bl\label{mainterm} Let $K$ be a number field of degree $n$. Let $\omega(\p)=1,  \forall \, \p\in\P$, and $\omega(\p)=0, \forall \, \p\in\overline{\P}$. Let $c=n^{10n^2}(\textup{Reg}_Kh_K)^{1/n}(1+ \log\textup{Reg}_Kh_K)^{\frac{(n-1)^2}{n}}$. Then we have, for all $z\geq (ec/\kappa_K)^n$, that\begin{equation*}G(z) \geq \prod_{\substack{\p\in\bar{\P}\\N(\p)\leq z}}\left(1-\f{1}{N(\p)}\right) \kappa_K \log\left(\frac{\kappa_K^n z}{e^nc^n}\right) .\end{equation*}
\el
\bp
Denote $\mathfrak{m}=\prod_{\substack{\p\in\overline{\P}\\N(\p)\leq z}}\p$. Then \begin{align*}G(z)= \sum_{\substack{N(\d)<z\\(\d,\m)=1}}\f{\mu^2(\d)}{\prod_{\p|\d}(N(\p)-1)} =\sum_{\substack{N(\d)<z\\(\d,\m)=1}}\f{\mu^2(\d)}{\phi(\d)}.\end{align*}
Note first that \ba G(z)\prod_{\substack{N(\q)<z \\ \q|\m}} \left(1-\f{1}{N(\q)}\right)^{-1} &=\sum_{\substack{(\d,\m)=1\\N(\d)\leq z}}\f{\mu^2(\d)}{\phi(\d)} \prod_{\substack{N(\q)<z \\ \q|\m}} \left(1-\f{1}{N(\q)}\right)^{-1}\\
&=\sum_{\substack{(\d,\m)=1\\N(\d)\leq z}}\f{\mu^2(\d)}{\phi(\d)} \prod_{\substack{N(\q)<z \\ \q|\m}} \left(1+\f{1}{N(\q) -1}\right)\\
&\geq\sum_{N(\d)\leq z}\f{\mu^2(\d)}{\phi(\d)}.\end{align*}
We may further reduce the sum\[\sum_{N(\d)\leq z}\f{\mu^2(\d)}{\phi(\d)}
= \sum_{N(\d)\leq z} \f{\mu^2(\d)}{N(\d)} \prod_{\p|\d} \left(1-\f{1}{N(\p)}\right)^{-1}
\geq \sum_{N(\a)\leq z}\f{1}{N(\a)}.\]
Let $S(x)=\sum_{m\leq x}a_m$. By partial summation,
\begin{align*}\sum_{m\leq z}\f{a_m}{m} 
&= \sum_{m\leq y}\f{a_m}{m} + \sum_{y\leq m\leq z-1}\frac{S(m)}{m(m+1)} + \frac{S(z-1)}{z-1} -\frac{S(y-1)}{y} \\
&\geq \kappa_K(\log(z)-\log(y+1)) - nc \frac{1}{(y+1)^{1/n}},
\end{align*}
where we have used Theorem \ref{explicitideals}. Choosing $y+1=(c/\kappa_K)^n$ yields the theorem.
\end{proof}
The following lemma deals with the error term in the case that $|R_\d| \ll \frac{1}{N(\d)^{1-\f{1}{n}}}$.
\bl\label{errorterm} Let $K$ be a number field of degree $n$. Then, for all $z\geq 16$,\begin{equation*}
\sum_{\substack{N(\d)\leq z^2\\ \d|\P(z)}}3^{\nu(\d)}\frac{1}{N(\d)^{1-\f{1}{n}}} \leq z^{2/n} (2\log z)^{3n}.
\end{equation*}
\el
\bp
Since a prime can split into at most $n$ prime ideals $\p$ in $\O_K$, with norms at least $p$, we have that \[\sum_{N(\p)\leq z}\f{1}{N(\p)}\leq n\sum_{p\leq z}\frac{1}{p}.\] 
Since by \cite{Dusart}, $p_k\geq k(\log k + \log\log k -1)\geq k\log k$ for $k\geq e^e=15,154...$. Hence, \[\sum_{p\leq z}\frac{1}{p}=\sum_{k\leq 15}\frac{1}{p_k}+\sum_{16\leq k\leq z}\frac{1}{p_k}\leq 0,666 + \log\log z,\]
for $z\geq 16$. We may conclude that
\begin{equation}\label{eq:eqprod}\prod_{N(\p)\leq z}\left(1+\f{1}{N(\p)}\right) 
\leq e^{\sum_{N(\p)\leq z}\f{1}{N(\p)}}
\leq \left(2\log z\right)^{n},\end{equation}
since $e^{0,666}\leq 2$. Now note that \begin{align*}\sum_{\substack{N(\d)\leq z^2\\ \d|\P(z)}}3^{\nu(\d)}\frac{1}{N(\d)^{1-\f{1}{n}}} & \leq z^{2/n} \sum_{\substack{N(\d)\leq z^2\\ \d|\P(z)}}3^{\nu(\d)}\frac{1}{N(\d)}\\
&\leq z^{2/n} \prod_{\substack{N(\p)\leq z\\\p\in\P}}\left(1+\frac{1}{N(\p)}\right)^3\\
&\leq z^{2/n} (2\log z)^{3n}. \qedhere
\end{align*}
\end{proof}

\section{Artin's Reciprocity law and proof of Theorem \ref{maintheorem}}
Let $\ff$ an integral ideal in $\O_K$. Recall the narrow class group to the modulus $\ff$, denoted by $H_\ff^{*}$. We will denote the order of this group by $h_\ff^{*}$. Consider the map \begin{align*}\rho : \{\alpha\in\O_K\mid (\alpha, \ff)=1\} &\rightarrow \big(\O_K/\ff \big)^{*} \times \{\pm1\}^{r_1}\\ 
\alpha & \mapsto ( \alpha + \ff , \sgn \sigma_1(\alpha),\dots , \sgn \sigma_{r_1}(\alpha)),
\end{align*}
and let $V_K$ be the group of all units in $\O_K$, including roots of unity. Then $(\alpha)$ is in the trivial class, denoted $\e$, of $H_\ff^{*}$ if and only if $\rho(\alpha)\in\rho(V_K)$, and we have that \[h_\ff^{*}=\frac{2^{r_1}\phi(\ff)}{|\rho(V_K)|}h_K,\]
moreover, each ideal class splits into \[\frac{2^{r_1}\prod_{\p|\ff} \big(1-\frac{1}{N(\p)}\big) N(\ff)}{|\rho(V_K)|}\] narrow classes mod $\ff$.
The narrow class groups mod $\ff$ are relevant in Class Field Theory, and in the Artin reciprocity law in particular. One possible statement of Artin reciprocity law reads as follows (see \cite[p. 389]{Narkiewicz}).
\bt(Artin Reciprocity Law) If $L/K$ is Abelian and $G$ is its Galois group, then there exists an ideal $\ff$ in $\O_K$ such that the prime ideals ramified under $L/K$ coincide with the prime ideals dividing $\ff$, and if $\p_1$ and $\p_2$ are two unramified prime ideals in the same class in $H_\ff^{*}$, then Frob$_{L/K}(\p_1) = $Frob$_{L/K}(\p_2)$. The map of $H_\ff^{*}$ into $G$ induced in this way by Frob$_{L/K}$ turns out to be a surjective homomorphism.
\et 
The ideal $\ff$ with smallest norm such that the above holds is called the conductor of $L/K$.
The strategy to count primes with a given Frobenius conjugacy class $C\subset Gal(L/\Q)$ can then be outlined in the following steps. \\
(a) Choose an abelian subgroup $H\leq G$ such that $C\cap H$ is nonempty. Let $K$ be the subfield corresponding to $H$. We reduce to the problem of counting prime ideals in $\O_K$ with a given Frobenius in the {\em abelian} extension $L/K$.\\
(b) Applying the Artin reciprocity law, we reduce this to the problem of counting prime ideals in given narrow ideal classes mod $\ff$.\\
(c) Using our explicit Selberg Sieve, this is reduced to the problem of counting ideals in given narrow ideal classes mod $\ff$, which is then achieved by means of our counting result, Theorem \ref{keylemma}.

The following proposition gathers all parts of the proof of Theorem \ref{maintheorem} which rely on the Selberg Sieve and the counting result, Theorem $\ref{keylemma}$.
\begin{prop}\label{sievingingr} Let $R$ be a subset of the narrow class group $H_\ff^{*}$. Let \[\mathcal{A} = \cup_{\c\in R}\{\a \textup{ integral ideal in }\O_K\mid \substack{N(\a)\leq x \\ \a\in\c}\},\] and let $\P$ consist of all primes coprime to $\ff$. Let $ z\geq (ec/\kappa_K)^n$ Then
\begin{align*}S(\A,\P,z) \leq \frac{R}{h_\ff^{*}} \frac{x}{\log\left(\frac{\kappa_K^n z}{e^nc^n}\right)} &+ N(\ff)^{1/n} cx^{1-\frac{1}{n}}z^{\f{2}{n}}(2\log z)^{3n}\\&\quad\quad+2^{r_1} N(\ff)(2rn)^{2r}h_K\textup{Reg}_Kz^{2}(2\log z)^{3n},\end{align*}
where $c=n^{10n^2}(\textup{Reg}_Kh_K)^{1/n}(1+ \log\textup{Reg}_Kh_K)^{\frac{(n-1)^2}{n}}$.
\end{prop}
\bp Let $\d$ be a squarefree ideal coprime with $\ff$, and consider the set $\{ \a \subset \d \mid \a \in \c, N(\a)\leq x\}$. Choose an ideal $\b \in \c^{-1}$. Then, this set has cardinality equal to \begin{align*}&|\{ (\alpha) \subset \d\b \mid (\alpha) \in \e, N(\alpha)\leq xN(\b)\}| \\
&= \frac{1}{\omega}|\{ \alpha \in \d\b \mid \rho(\alpha) \in \rho(V_K), \alpha \in \F(xN(\b))\}| \\
&= \frac{1}{\omega}\sum_{k=1}^{|\rho(V_K)|}|\{ \alpha \in \d\b \mid \alpha \equiv u_k \bmod \ff, \alpha \in  \F_{(\sgn\sigma_i(u_k))}(xN(\b))\}|,
\end{align*}
where the $u_k$ are such that $\rho(V_K)=\{\rho(u_k)\mid k=1,\dots, |\rho(V_K)|\}$.
By Theorem \ref{keylemma}, this is \[\frac{|\rho(V_K)|}{\omega}\left( \frac{\omega \kappa_K}{2^{r_1}N(\ff)h_K} \frac{x}{N(\d)} + O^{*}\big(n^{\frac{3}{2}n^2}e^{4n^2} \NN(\frC) \frac{x^{\frac{n-1}{n}}}{N(\d\ff)^{\frac{n-1}{n}}}+m_1\cdots m_r\big)\right),\]
where $\frC$ is the ideal class of $(\b\ff\d)^{-1}$. We sum this over $R$ classes $\c$ to get an estimate for $|\A_\d|$. As a main term we get \[\frac{R\prod_{\p|\ff} \big(1-\frac{1}{N(\p)}\big)\kappa_K}{h_\ff^{*}}\frac{x}{N(\d)}.\] We bound the error term above by summing over all $h_{\ff}^{*}$ classes, and we get \begin{align*}|R_\d| &\leq \frac{2^{r_1} N(\ff)}{\omega}n^{\frac{3}{2}n^2}e^{4n^2} \left(\sum_{i=1}^{h_Km_1\dots m_r} \frac{1}{N(\b)^{\frac{n-1}{n}}}\right) \frac{x^{\frac{n-1}{n}}}{N(\d\ff)^{\frac{n-1}{n}}}+2^{r_1} N(\ff)h_Km_1\cdots m_r\\
& \leq N(\ff)^{1/n}c\frac{x^{\frac{n-1}{n}}}{N(\d)^{\frac{n-1}{n}}}+2^{r_1} N(\ff)(2rn)^{2r}h_K\textup{Reg}_K,
\end{align*}
by applying Lemma \ref{boundfinally} in the same way as in the proof of Theorem \ref{explicitideals}. The statement now follows by applying Selberg's Sieve, Theorem \ref{selberg}, and Lemmata \ref{mainterm} and \ref{errorterm}.
\end{proof}
\begin{proof}[Proof of Theorem \ref{maintheorem}]
Let $H\leq Gal(L/\Q)$ be the maximal abelian subgroup such that $H\cap C$ is non-empty. Let $K$ be the fixed field of $H$, and choose $c\in C\cap H$. Let $n$ be the degree of $K$. It is an exercise in basic algebraic number theory to prove that
\[\pi_C(x, L/\Q) \leq \frac{|C|}{n} |\{\p \textup{ unramified prime ideal in } \O_K \mid \textup{Frob}_{L/K}(\p)= c, N(\p)\leq x\}|.\]
By the Artin reciprocity law, there is a subset $R\subseteq H_\ff^{*}$ of size $\f{h_\ff^{*}}{[L:K]}$ such that \[\pi_C(x,L/\Q) \leq \frac{|C|}{n}\sum_{\c \in R}|\{\p \textup{ prime ideal in } \O_K \mid \p\in \c, N(\p)\leq x\}|.\]
By Proposition \ref{sievingingr}, we may bound this by \begin{align*} \frac{|C|}{n} \Big(&\frac{1}{[L:K]}\frac{x}{\log\left(\frac{\kappa_K^n z}{e^nc^n}\right)} +  N(\ff)^{1/n} cx^{1-\frac{1}{n}}z^{\f{2}{n}}(2\log z )^{3n}\\
&+2^{r_1} N(\ff)(2rn)^{2r}h_K\textup{Reg}_Kz^{2}(2\log z)^{3n}+n\pi(z)\Big),\end{align*}
where $c=n^{10n^2}(\textup{Reg}_Kh_K)^{1/n}(1+ \log\textup{Reg}_Kh_K)^{\frac{(n-1)^2}{n}}$, for any $z\geq (ec/\kappa_K)^n$.
We choose $z^2=\frac{x}{(\log x)^{4n^2} (c[L:K])^{n}N(\ff)}$, which is a valid choice since $\frac{x}{(\log x)^{5n^2}}\geq e^{2n}c^{3n}{[L:K]^n}N(\ff)\frac{1}{\kappa_K^{2n}}$. This choice of $z$ makes sure that \[\left(N(\ff)^{1/n} cx^{1-\frac{1}{n}}z^{\f{2}{n}}+2^{r_1} N(\ff)(2rn)^{2r}h_K\textup{Reg}_Kz^{2}\right)(2\log z)^{3n}+n\pi(z) \leq \frac{x}{[L:K](\log x )^2},\]
while \[\frac{x}{\log\left(\frac{\kappa_K^n z}{e^nc^n}\right)} \leq \frac{2x}{\log(x/c_K)  - 4n^2\log_2 x},\]
where $c_K= n^{31n^3}(\textup{Reg}_Kh_K)^3(1+\log \textup{Reg}_Kh_K)^{3(n-1)^2} [L:K]^n N(\ff) \kappa_K^{-2n}$.
Now \begin{align*}&\frac{1}{\log(x/c_K)  - 4n^2\log_2 x} \\ & = \frac{1}{\log(x/c_K)} + \frac{4n^2\log_2 x}{\log(x/c_K)\big(\log(x/c_K)-4n^2\log_2 x\big)}\\
& \leq \frac{1}{\log(x/c_K)} +  \frac{20n^2\log_2 x}{\log(x/c_K)^2},
\end{align*}
since $(\log x)^{5n^2}\leq x/c_K$ by assumption. In conclusion, we have bounded $\pi_C(x,L/\Q)$ above by \[\frac{|C|}{|G|}\left(\frac{2x}{\log(x/c_K)} +  \frac{20n^2\log_2 x}{\log(x/c_K)}\frac{2x}{\log(x/c_K)} + \frac{x}{(\log x)^2}\right).\qedhere\]
\end{proof}

\bibliographystyle{siam}
\bibliography{biblio}

\end{document}